\documentclass{amsart}
\usepackage{graphicx}
\newtheorem{Thm}{Theorem}
\newtheorem{Lem}[Thm]{Lemma}
\newtheorem{Cor}[Thm]{Corollary}
\newtheorem{Conj}[Thm]{Conjecture}
\newtheorem{Prob}[Thm]{Problem}

\theoremstyle{definition}

\theoremstyle{remark}

\newtheoremstyle{component}{}{}{}{}{\itshape}{.}{.5em}{\thmnote{#3}#1}
\theoremstyle{component}
\newtheorem*{component}{}

\begin{document}

\title[Heegaard splittings and singularities]{Heegaard splittings and singularities of the product map of Morse functions}
\author{Kazuto Takao}
\address{Department of Mathematics, Graduate School of Science, Hiroshima University, Higashi-Hiroshima 739-8526 Japan}
\email{kazutotakao@gmail.com}

\begin{abstract}
We give an upper bound for the Reidemeister--Singer distance between two Heegaard splittings in terms of the genera and the number of cusp points of the product map of Morse functions for the splittings.
It suggests that a certain development in singularity theory may lead to the best possible bound for the Reidemeister--Singer distance.
\end{abstract}

\maketitle

\section{Introduction}

A {\it Heegaard splitting} is a decomposition of a closed orientable connected $3$-manifold into two handlebodies.
Any such $3$-manifold admits a Heegaard splitting but it is not unique.
An important problem in $3$-manifold topology is the classification of the Heegaard splittings for a $3$-manifold.
As part of the classification problem, it is of particular interest how we can estimate the {\it Reidemeister--Singer distance} between two Heegaard splittings for a $3$-manifold.
(See Section \ref{Heegaard} for definitions and known results.)

Heegaard splittings are closely related to Morse functions on the manifold.
We can choose a Morse function which represents a given Heegaard splitting.
(See Section \ref{Morse} for the details.)
Suppose $F,G:M\rightarrow {\mathbb R}$ are Morse functions representing Heegaard splittings $(\Sigma ,V^-,V^+),(T,W^-,W^+)$, respectively, for a given fixed $3$-manifold $M$.
We assume that each of them has unique critical points of indices zero and three, and that the product map $F\times G$ is {\it stable}.
In this paper we show the following:

\begin{Thm}\label{main-1}
The Reidemeister--Singer distance between $(\Sigma ,V^-,V^+)$, $(T,W^-,W^+)$ is at most ${\rm g}(\Sigma )+{\rm g}(T)+c(F\times G)/2$.
\end{Thm}

\noindent
Here, ${\rm g}(\cdot )$ denotes the genus of a surface, and $c(\cdot )$ denotes the number of {\it cusp points} of a stable map.
(See Section \ref{Stable} for notions on singularities.)

This work is based on the idea of the so-called {\it Rubinstein--Scharlemann graphic}.
Rubinstein--Scharlemann \cite{rubinstein-scharlemann1} constructed it via Cerf theory \cite{cerf} from {\it sweep-outs} (see \cite{rubinstein-scharlemann1} for the definition) which come from two Heegaard splittings.
Kobayashi--Saeki \cite{kobayashi-saeki} redefined a sweep-out as a sort of function on the manifold, and interpreted the graphic as the discriminant set of the product map of the sweep-outs representing the Heegaard splittings.
Johnson \cite{johnson1} constructed the graphic from Morse functions instead of sweep-outs, and showed the following:

\begin{Thm}[Johnson]\label{johnson}
The Reidemeister--Singer distance between $(\Sigma ,V^-,V^+)$, $(T,W^-,W^+)$ is at most $i_-+c_2$.
Here, $i_-$ is the number of indefinite negative slope inflection points and $c_2$ is the number of negative slope type two cusps of the graphic of $F$ and $G$.
\end{Thm}

\noindent
We remark that $c_2$ turns out to be zero and so the Reidemeister--Singer distance is at most $i_-$.
(See Section \ref{Reading} for an analysis of the graphic.)

Note that $i_-$ is not an invariant of the Heegaard splittings but depends on the choice of $F$ and $G$.
To get a bound for the Reidemeister--Singer distance in terms of invariants of the Heegaard splittings, the following problem is proposed.

\begin{Prob}
Find an upper bound for the minimum of $i_-$ over all choices of $F$ and $G$ in terms of invariants of the Heegaard splittings.
\end{Prob}

This seems a complicated problem as Johnson mentioned in \cite{johnson1}.
While $c(F\times G)$ unfortunately also depends on the choice of $F$ and $G$, Theorem \ref{main-1} gives a reinterpretation of the problem as follows:

\begin{Prob}
Find an upper bound for the minimum of $c(F\times G)$ over all choices of $F$ and $G$ in terms of invariants of the Heegaard splittings.
\end{Prob}

This might be easier for singularity theoretic approaches.
Significantly, it is known by Levine \cite{levine1} that cusp points of a stable map from an orientable $3$-manifold to an orientable surface can be eliminated by a homotopy of the map.
Theorem \ref{main-1} and Levine's theorem suggest the following:

\begin{Conj}
The Reidemeister--Singer distance between two Heegaard splittings is at most the sum of the genera.
\end{Conj}

The Hass--Thompson--Thurston example \cite{hass-thompson-thurston} proves that the sum of the genera is the best possible bound.
The conjecture however cannot be proved immediately because Levine's homotopy does not preserve the product structure of $F\times G$.
 
\subsection*{Acknowledgement}

The author is grateful to \mbox{Ken'ichi Ohshika} for all his help as a supervisor.
He would like to thank \mbox{Tsuyoshi Kobayashi} and \mbox{Osamu Saeki} for valuable discussions and encouragement.
He would also like to thank the referee for many corrections and suggestions.

\section{Heegaard surfaces and splittings}\label{Heegaard}

It is well known that any closed orientable connected $3$-manifold $M$ can be decomposed into two handlebodies $V^-,V^+$ by an embedded surface $\Sigma $.
In this situation, we call the surface $\Sigma $ a {\it Heegaard surface} for $M$ and the ordered triple $(\Sigma ,V^-,V^+)$ a {\it Heegaard splitting} for $M$.
Two Heegaard surfaces $\Sigma ,T$ for a $3$-manifold are said to be {\it isotopic} if there is an ambient isotopy of $M$ taking $\Sigma $ to $T$.
In this paper, two Heegaard splittings $(\Sigma ,V^-,V^+),(T,W^-,W^+)$ are said to be {\it isotopic} if an isotopy from $\Sigma $ to $T$ takes $V^-$ to $W^-$ and simultaneously $V^+$ to $W^+$.
Note that the flipped Heegaard splitting $(\Sigma ,V^+,V^-)$ is not always isotopic to $(\Sigma ,V^-,V^+)$, though their Heegaard surfaces are the same.
We frequently mean its isotopy class by a Heegaard surface/splitting.

A {\it stabilization} for $(\Sigma ,V^-,V^+)$ is an operation to obtain another Heegaard splitting $(\Sigma ',V^{\prime -},V^{\prime +})$ for $M$ as follows:
Suppose $\gamma $ is a properly embedded arc in $V^+$ parallel to $\Sigma $ and $N(\gamma )$ denotes its closed regular neighborhood.
Let $V^{\prime -}$ be the union $V^-\cup N(\gamma )$, $V^{\prime +}$ be the closure of $V^+\setminus N(\gamma )$ and $\Sigma '$ be their common boundary.
Both $V^{\prime -},V^{\prime +}$ are then handlebodies and so $(\Sigma ',V^{\prime -},V^{\prime +})$ is a Heegaard splitting for $M$.
Since a stabilization increases the genus of both handlebodies by one, it defines a {\it stabilization} for the Heegaard surface $\Sigma $.
Note that, for a given Heegaard surface/splitting, a stabilization always gives a unique Heegaard surface/splitting up to isotopy.
On the other hand the inverse operation, which is called a {\it destabilization}, is not always possible and not always unique.

It is known by Reidemeister \cite{reidemeister} and Singer \cite{singer} that any two isotopy classes of Heegaard surfaces for the same $3$-manifold are related by a finite sequence of stabilizations and destabilizations.
The same is also true for Heegaard splittings.
The {\it Reidemeister--Singer distance} between two Heegaard surfaces/splittings is defined by the minimal number of stabilizations and destabilizations relating them.
Note that the distance between $\Sigma ,T$ is the minimum of the distance between $(\Sigma ,V^-,V^+),(T,W^-,W^+)$ and the distance between $(\Sigma ,V^-,V^+),(T,W^+,W^-)$.
For both Heegaard surfaces and splittings, the question of how we can estimate the distance is called {\it the Stabilization Problem}.

Rubinstein--Scharlemann \cite{rubinstein-scharlemann1} showed that the distance between two Heegaard surfaces $\Sigma ,T$ with ${\rm g}(\Sigma )\leq {\rm g}(T)$ for a non-Haken manifold is at most $9{\rm g}(\Sigma )+15{\rm g}(T)-18$.
In fact the same bound also stands for Heegaard splittings.
Recently, Johnson \cite{johnson3} showed that the distance between two Heegaard splittings $(\Sigma ,V^-,V^+),(T,W^-,W^+)$ with ${\rm g}(\Sigma )\leq {\rm g}(T)$ for a $3$-manifold is at most $3{\rm g}(\Sigma )+2{\rm g}(T)-2$.
On the other hand, Hass--Thompson--Thurston \cite{hass-thompson-thurston} gave a Heegaard splitting $(\Sigma ,V^-,V^+)$ whose Reidemeister--Singer distance to $(\Sigma ,V^+,V^-)$ is $2{\rm g}(\Sigma )$.
For every integer $k\geq 2$, Johnson \cite{johnson2} gave a $3$-manifold with Heegaard surfaces of genera $2k-1$ and $2k$ such that the distance between them is $2k-1$.
Bachman \cite{bachman} also gave these kinds of examples of Heegaard surfaces and splittings.
For every integer $k\geq 2$, the author \cite{takao} modified Johnson's arguments to give a $3$-manifold with two Heegaard surfaces of genus $2k$ such that the distance between them is $2k$.

As well as Theorem \ref{main-1} for Heegaard splittings, we also show the following for Heegaard surfaces under the same assumption.

\begin{Thm}\label{main-2}
The Reidemeister--Singer distance between $\Sigma ,T$ is at most ${\rm g}(\Sigma )+{\rm g}(T)+c(F\times G)/4$.
\end{Thm}

\section{Morse functions}\label{Morse}

In this section, we briefly describe the connection between Heegaard splittings and Morse functions.
We define a {\it Morse function} as a smooth function whose critical points are all non-degenerate.
(See \cite{milnor} for basic notions in Morse theory.)
Note that we do not assume that the critical points have pairwise distinct values.
The condition can be satisfied after an arbitrarily small {\it quasi-isotopy}.
A {\it quasi-isotopy} is a homotopy consisting of Morse functions.
It is an equivalence relation of Morse functions and each equivalence class is called a {\it quasi-isotopy class}.

Suppose $M$ is a closed orientable connected smooth $3$-manifold and $F:M\rightarrow {\mathbb R}$ is a Morse function.
Let $\tilde{F}$ be a Morse function obtained from $F$ by a small quasi-isotopy such that the critical points of $\tilde{F}$ have pairwise distinct values.
Such a Morse function corresponds to a handle decomposition as follows:
Let $c_1,c_2,\ldots ,c_n$ be the critical values and $r_0,r_1,\ldots ,r_n$ be regular values of $\tilde{F}$ such that $r_0<c_1<r_1<c_2<\cdots <c_n<r_n$.
Each $\tilde{F}^{-1}((-\infty ,r_j])$ is isotopic to the result of attaching an $i_j$-handle to $\tilde{F}^{-1}((-\infty ,r_{j-1}])$, where $i_j$ is the index of the critical point at $c_j$.
In this way, $\tilde{F}$ determines a handle decomposition for $M$.
Conversely, we can construct a Morse function corresponding to a given handle decomposition.

A quasi-isotopy keeps all the critical points non-degenerate but does not always keep the order of the critical values in ${\mathbb R}$.
A quasi-isotopy changing the order of the critical values corresponds to a rearrangement of handles of the handle decomposition.
It is well known that the handles can be rearranged so that the lower the index is, the earlier a handle is attached.
Then the union $V^-$ of the $0$-handles and the $1$-handles is a handlebody, and dually the union $V^+$ of the $2$-handles and the $3$-handles is also a handlebody.
Denoting the common boundary of $V^-,V^+$ by $\Sigma $, the triple $(\Sigma ,V^-,V^+)$ is a Heegaard splitting for $M$.
It is known that the isotopy class of the Heegaard splitting $(\Sigma ,V^-,V^+)$ is independent of the order in which the handles are attached (see \cite{schultens}).
We thus have the following:

\begin{Lem}\label{uniqueness}
A quasi-isotopy class of Morse functions determines a unique isotopy class of Heegaard splittings.
\end{Lem}

We remark that the condition about critical values is not required for a Morse function to determine an isotopy class of Heegaard splittings.
In this sense, we say that $F$ {\it represents} $(\Sigma ,V^-,V^+)$.
Note that the Morse function $-F$ represents the flipped splitting $(\Sigma ,V^+,V^-)$.

Regarding a Heegaard splitting as a handle decomposition, a stabilization is adding a canceling pair of a $1$-handle and a $2$-handle.
On the other hand, we can see that adding a canceling pair of a $0$-handle and a $1$-handle does not change the isotopy class of the Heegaard splitting.
Interpreting them in terms of Morse functions, we have the following:

\begin{Lem}\label{stabilization}
A birth/death of a canceling pair of critical points of indices $1$ and $2$ for a Morse function causes a stabilization/destabilization for the represented Heegaard splitting.
On the other hand, one of indices $0$ and $1$ or indices $2$ and $3$ preserves the Heegaard splitting.
\end{Lem}

\section{Stable maps}\label{Stable}

To study the Stabilization Problem, we should address the relationship between two Heegaard splittings in a $3$-manifold.
Suppose $(\Sigma ,V^-,V^+),(T,W^-,W^+)$ are Heegaard splittings for a $3$-manifold $M$.
Instead of comparing two Heegaard splittings themselves, we may compare corresponding Morse functions in the sense of Lemma \ref{uniqueness}.
Let $F,G:M\rightarrow {\mathbb R}$ be Morse functions representing $(\Sigma ,V^-,V^+)$, $(T,W^-,W^+)$, respectively.
A method of comparing two Morse functions is analyzing their product.
The {\it product map} $F\times G$ is the map from $M$ to ${\mathbb R}^2$ defined by $(F\times G)(p)=(F(p),G(p))$.
Then, in this section, we briefly review basic notions and facts on singularities of smooth maps from $3$-manifolds to $2$-manifolds.

Suppose $\varphi $ is a smooth map from a closed orientable smooth $3$-manifold $M$ to an orientable smooth $2$-manifold $N$.
Recall that $p\in M$ is a {\it regular point} of $\varphi $ if the differential $(d\varphi )_p:T_pM\rightarrow T_{\varphi (p)}N$ is surjective, and otherwise is a {\it singular point}.
The set $S_\varphi $ of singular points of $\varphi $ is called the {\it singular set} and its image $\varphi (S_\varphi )$ is called the {\it discriminant set} of $\varphi $.
At a regular point $p\in M\setminus S_\varphi $, the map $\varphi $ has the standard form $\varphi (u,x,y)=(u,x)$ for some coordinate neighborhoods of $p$ and $\varphi (p)$.
Standard forms are also known for generic types of singular points as follows.

A {\it definite fold point} is a singular point $p$ such that $\varphi $ has the form $\varphi (u,x,y)=(u,x^2+y^2)$ for a coordinate neighborhood $U$ of $p=(0,0,0)$ and a coordinate neighborhood of $\varphi (p)=(0,0)$.
The Jacobian matrix of $\varphi (u,x,y)=(u,x^2+y^2)$ says that the singular set $S_\varphi \cap U$ is the arc $\{ (u,0,0)\} $.
It follows that each singular point on $S_\varphi \cap U$ is also a definite fold point by a translation of the local coordinates.
The arc $\{ (u,0,0)\} $ is embedded to the arc $\{ (u,0)\} \subset N$ by $\varphi $.

An {\it indefinite fold point} is a singular point $p$ where $\varphi $ has the local form $\varphi (u,x,y)=(u,x^2-y^2)$.
Similarly, the singular set $S_\varphi \cap U$ is the arc $\{ (u,0,0)\} $, consists of indefinite fold points and is embedded to the arc $\{ (u,0)\} \subset N$.
Collectively definite and indefinite fold points are referred to as {\it fold points}.

A {\it cusp point} is a singular point $p$ where $\varphi $ has the local form $\varphi (u,x,y)=(u,y^2+ux-x^3)$.
The singular set $S_\varphi \cap U$ is the arc $\{ (3x^2,x,0)\} $ centered at the point $p=(0,0,0)$.
One can check that the half $\{ (3x^2,x,0)\mid x<0\} $ consists of definite fold points and the other half $\{ (3x^2,x,0)\mid x>0\} $ consists of indefinite fold points.
Note that the arc $\{ (3x^2,x,0)\} $ is a regular curve but its image $\{ (3x^2,2x^3)\} \subset N$ has a {\it cusp} at $\varphi (p)=(0,0)$.
A {\it cusp} of a smooth curve on a surface is a point at which the two branches of the curve have a common tangent line and lie on the same side of the normal line.

Assume that every singular point of $\varphi $ is a fold point or a cusp point.
Then, by the above local observations and the compactness of $M$, we can see the outline of the singular set $S_\varphi $.
It is a $1$-dimensional submanifold of $M$, namely a collection of smooth circles.
There are finitely many cusp points and the restriction $\varphi |_{S_\varphi }$ is an immersion except that cusp points map to cusps.

\begin{Thm}[Mather \cite{mather5}]
A smooth map $\varphi :M\rightarrow N$ is stable if and only if:
\begin{enumerate}
\item The singular set $S_\varphi $ consists only of fold points and cusp points.
\item The restriction $\varphi |_{S_\varphi }$ has no double points on the cusps, and the immersion $\varphi |_{S_\varphi \setminus \text{(cusp points)}}$ has normal crossings.
\end{enumerate}
\end{Thm}

A {\it stable} map is, roughly speaking, a smooth map which does not change topologically by small deformations.
Strictly speaking, a smooth map $\varphi :M\rightarrow N$ between smooth manifolds $M,N$ is called {\it stable} if there exists an open neighborhood $U$ of $\varphi $ in $C^\infty (M,N)$ such that every map in $U$ can be obtained by an {\it isotopy} of $\varphi $.
Here, we denote by $C^\infty (M,N)$ the space of smooth maps from $M$ to $N$ endowed with the Whitney $C^\infty $ topology (see \cite{golubitsky-guillemin} or \cite{hirsch}).
An {\it isotopy} of $\varphi =\varphi _0$ is a homotopy $\{ \varphi _r\} _{r\in [0,1]}$ which is decomposed as $\varphi _r=H^N_r\circ \varphi _0\circ H^M_r$, where $\{ H^M_r\} _{r\in [0,1]},\{ H^N_r\} _{r\in [0,1]}$ are smooth ambient isotopies of $M,N$, respectively.
A stable map to $\mathbb{R}$ is a Morse function whose critical points have pairwise distinct values.

Now, we return to comparing two Heegaard splittings.
We constructed the product map $F\times G:M\rightarrow {\mathbb R}^2$ of two Morse functions $F,G$.
The following lemma is fundamental to our arguments.

\begin{Lem}\label{stable}
We can make the product map $F\times G$ stable by arbitrarily small quasi-isotopies of $F$ and $G$.
\end{Lem}

\begin{proof}
We show that in arbitrarily open neighborhoods in $C^\infty (M,{\mathbb R})$ of $F,G$, there exist Morse functions $F',G'$ quasi-isotopic to $F,G$, respectively, such that $F'\times G'$ is stable.

Let $U_F,U_G\subset C^\infty (M,{\mathbb R})$ be open neighborhoods of $F,G$, respectively.
We can make $F$ (resp. $G$) stable by an arbitrarily small quasi-isotopy, that is, there exists a stable function $\tilde{F}$ in $U_F$ (resp. $\tilde{G}$ in $U_G$) quasi-isotopic to $F$ (resp. $G$).
By the definition of a stable map, there exists an open neighborhood $U_{\tilde{F}}$ of $\tilde{F}$ (resp. $U_{\tilde{G}}$ of $\tilde{G}$) in $C^\infty (M,{\mathbb R})$ such that every function in $U_{\tilde{F}}$ (resp. $U_{\tilde{G}}$) is isotopic to $\tilde{F}$ (resp. $\tilde{G}$) and hence quasi-isotopic to $F$ (resp. $G$).

Note that the plane ${\mathbb R}^2$ has been coordinated so that $(F\times G)(p)=(F(p),G(p))$.
Let $\pi _f,\pi _g:{\mathbb R}^2\rightarrow {\mathbb R}$ be the projections $(f,g)\mapsto f$, $(f,g)\mapsto g$, respectively.
Since $\pi _f$ is smooth, the induced map $\pi _{f*}:C^\infty (M,{\mathbb R}^2)\rightarrow C^\infty (M,{\mathbb R})$, $\varphi \mapsto \pi _f\circ \varphi $ is continuous by \cite[Chapter I\hspace{-.1em}I, Proposition 3.5]{golubitsky-guillemin}.
Similarly, the map $\pi _{g*}$ induced by $\pi _g$ is continuous.

The subset $\pi _{f*}^{-1}(U_F\cap U_{\tilde{F}})\cap \pi _{g*}^{-1}(U_G\cap U_{\tilde{G}})$ of $C^\infty (M,{\mathbb R}^2)$ is open.
Since stable maps are dense in $C^\infty (M,{\mathbb R}^2)$ by Mather \cite{mather5}, there exists a stable map $\varphi $ in $\pi _{f*}^{-1}(U_F\cap U_{\tilde{F}})\cap \pi _{g*}^{-1}(U_G\cap U_{\tilde{G}})$.
The functions $F'=\pi _{f*}(\varphi )$, $G'=\pi _{g*}(\varphi )$ satisfy the claim.
\end{proof}

\section{Reading the graphic}\label{Reading}

We would like to read information about two smooth functions $F,G:M\rightarrow {\mathbb R}$ from the product map $\varphi :=F\times G$.
While we do not assume $F,G$ to be Morse in this section, we assume $\varphi $ to be stable throughout the following.
We call the discriminant set of $\varphi $ the {\it graphic} of $F$ and $G$.

There is a global coordinate system $(f,g)$ of ${\mathbb R}^2$ with respect to which $F\times G$ is defined as $(F\times G)(p)=(F(p),G(p))$.
The Jacobian matrix of $\varphi $ with respect to the coordinate system is composed of the gradients of $F$ and $G$.
If $p$ is a critical point of $F$ or $G$, the differential $(d\varphi )_p$ has rank at most one, namely $p$ is a singular point of $\varphi $.
So the singular set $S_\varphi $ includes all the critical points of $F$ and $G$.

Since we assume $\varphi $ to be stable, the map $\varphi |_{S_\varphi }$ is an immersion of circles with finitely many cusps.
We can therefore define the {\it slope} of the graphic $\varphi (S_\varphi )$ at $\varphi (p)$ for each $p\in S_\varphi $ with respect to the coordinate system $(f,g)$.
In particular, a point on the graphic with slope zero (resp. infinity) is called a {\it horizontal} (resp. {\it vertical}) {\it point}.
Such points correspond to critical points of $F$ and $G$ as follows:

\begin{Lem}[Johnson]\label{critical}
A point $p\in M$ is a critical point of $G$ if and only if $\varphi (p)$ is a horizontal point of the image of a small neighborhood of $p$ in $S_\varphi $.
The same holds for $F$ by replacing ``horizontal" with ``vertical".
\end{Lem}

\noindent
The proof has been given in \cite[Lemma 10]{johnson1}, but we include it as a warm-up exercise for the proofs of the lemmas below.

\begin{proof}
Since a regular point of $\varphi $ cannot be a critical point of $G$, we suppose $p$ is a singular point of $\varphi $.
There are a coordinate system $(u,x,y)$ of a small neighborhood $U$ of $p=(0,0,0)$ and a local coordinate system $(s,t)$ at $\varphi (p)=(0,0)$ such that either $(s,t)=(u,x^2\pm y^2)$ if $p$ is a fold point, or $(s,t)=(u,y^2+ux-x^3)$ if $p$ is a cusp point.
On the other hand, the global coordinate system $(f,g)$ of ${\mathbb R}^2$ satisfies $(f,g)=(F(u,x,y),G(u,x,y))$.
There exists a smooth regular coordinate transformation from $(s,t)$ to $(f,g)$.

\begin{component}[Fold Case]
Consider the case where $p$ is a fold point of $\varphi $.
In order to compute the gradient of $G$, we compute the partial derivatives of $G$.
For instance,
$$\frac{\partial G}{\partial u}=\frac{\partial s}{\partial u}\frac{\partial g}{\partial s}+\frac{\partial t}{\partial u}\frac{\partial g}{\partial t}=\frac{\partial }{\partial u}(u)\frac{\partial g}{\partial s}+\frac{\partial }{\partial u}(x^2\pm y^2)\frac{\partial g}{\partial t}=\frac{\partial g}{\partial s}.$$
Similarly we have:
\begin{align*}
\frac{\partial G}{\partial x}&=2x\frac{\partial g}{\partial t},
&\frac{\partial G}{\partial y}&=\pm 2y\frac{\partial g}{\partial t}.
\end{align*}
Substituting $(u,x,y)=(0,0,0)$, the gradient vector of $G$ at $p$ is $((\frac{\partial g}{\partial s})_{\varphi (p)},0,0)$.
The point $p$ is a critical point of $G$ if and only if the coordinate transformation satisfies $(\frac{\partial g}{\partial s})_{\varphi (p)}=0$.
It means that the $s$-axis is parallel to the $f$-axis at $\varphi (p)$.
Recall that the singular set $S_\varphi \cap U$ is embedded to the $s$-axis by $\varphi $.
The image $\varphi (S_\varphi \cap U)$ is therefore horizontal at $\varphi (p)$.
\end{component}

\begin{component}[Cusp Case]
In the case where $p$ is a cusp point of $\varphi $, we have the partial derivatives:
\begin{align*}
\frac{\partial G}{\partial u}&=\frac{\partial g}{\partial s}+x\frac{\partial g}{\partial t},
&\frac{\partial G}{\partial x}&=(u-3x^2)\frac{\partial g}{\partial t},
&\frac{\partial G}{\partial y}&=2y\frac{\partial g}{\partial t}.
\end{align*}
The gradient vector at $p$ is $((\frac{\partial g}{\partial s})_{\varphi (p)},0,0)$.
The point $p$ is a critical point if and only if $(\frac{\partial g}{\partial s})_{\varphi (p)}=0$.
It means that the $s$-axis is parallel to the $f$-axis at $\varphi (p)$.
Note that the $s$-axis is the tangent line of $\varphi (S_\varphi \cap U)=\{ (s,t)=(3x^2,2x^3)\} $ at the cusp $\varphi (p)$.
The image $\varphi (S_\varphi \cap U)$ is therefore horizontal at $\varphi (p)$.
\end{component}
\end{proof}

We can also define the second derivative of the graphic outside of vertical points and cusps.
In particular, a point with second derivative zero is called an {\it inflection point}.
In fact, zero or non-zero of the second derivative is preserved by rotating the coordinate system.
Inflection can therefore also be defined for vertical points, but we still assume an inflection point not to be a cusp.

\begin{Lem}\label{degeneracy}
A critical point $p$ of $G$ degenerates if and only if $\varphi (p)$ is a horizontal inflection point of the image of a small neighborhood of $p$ in $S_\varphi $.
The same holds for $F$ by replacing ``horizontal" with ``vertical".
\end{Lem}

\noindent
This lemma in the case of fold points has been stated in \cite[Lemma 11]{johnson1} and we give a proof.
For the case of cusp points, we point out that no cusp point of $\varphi $ can be a degenerate critical point of $F,G$.

\begin{proof}
We continue with the notation in the proof of the previous lemma.
Suppose $p$ is a critical point of $G$, and hence $(\frac{\partial g}{\partial s})_{\varphi (p)}=0$.
The regularity of the coordinate transformation requires
$$\left| \begin{array}{cc}
\frac{\partial f}{\partial s}& \frac{\partial f}{\partial t}\\
\frac{\partial g}{\partial s}& \frac{\partial g}{\partial t}
\end{array}\right| 
=\frac{\partial f}{\partial s}\frac{\partial g}{\partial t}-\frac{\partial f}{\partial t}\frac{\partial g}{\partial s}\not=0.$$
So we have $(\frac{\partial f}{\partial s})_{\varphi (p)}(\frac{\partial g}{\partial t})_{\varphi (p)}\not=0$, and hence $(\frac{\partial f}{\partial s})_{\varphi (p)}\not=0$, $(\frac{\partial g}{\partial t})_{\varphi (p)}\not=0$.

\begin{component}[Fold Case]
Consider the case where $p$ is a fold point.
In order to compute the Hessian of $G$, we compute the second partial derivatives of $G$.
For instance,
\begin{align*}
\frac{\partial ^2G}{\partial x^2}&=\frac{\partial }{\partial x}\frac{\partial G}{\partial x}\\
&=\frac{\partial }{\partial x}\left( 2x\frac{\partial g}{\partial t}\right) \\
&=\frac{\partial }{\partial x}(2x)\frac{\partial g}{\partial t}+2x\frac{\partial }{\partial x}\left( \frac{\partial g}{\partial t}\right) \\
&=2\frac{\partial g}{\partial t}+2x\left( \frac{\partial s}{\partial x}\frac{\partial }{\partial s}+\frac{\partial t}{\partial x}\frac{\partial }{\partial t}\right) \left( \frac{\partial g}{\partial t}\right) \\
&=2\frac{\partial g}{\partial t}+2x\left( 2x\frac{\partial }{\partial t}\right) \left( \frac{\partial g}{\partial t}\right) \\
&=2\frac{\partial g}{\partial t}+4x^2\frac{\partial ^2g}{\partial t^2}.
\end{align*}
Similarly we have:
\begin{align*}
\frac{\partial ^2G}{\partial u^2}&=\frac{\partial ^2g}{\partial s^2},
&\frac{\partial ^2G}{\partial u\partial x}&=2x\frac{\partial ^2g}{\partial s\partial t},
&\frac{\partial ^2G}{\partial u\partial y}&=\pm 2y\frac{\partial ^2g}{\partial s\partial t},\\
\frac{\partial ^2G}{\partial x\partial y}&=\pm 4xy\frac{\partial ^2g}{\partial t^2},
&\frac{\partial ^2G}{\partial y^2}&=\pm 2\frac{\partial g}{\partial t}+4y^2\frac{\partial ^2g}{\partial t^2}.
\end{align*}
Substituting $(u,x,y)=(0,0,0)$, the Hessian matrix of $G$ at $p$ is
\begin{align*}
\left( \begin{array}{ccc}
\left( \frac{\partial ^2g}{\partial s^2}\right) _{\varphi (p)}& 0& 0\\
0& 2\left( \frac{\partial g}{\partial t}\right) _{\varphi (p)}& 0\\
0& 0& \pm 2\left( \frac{\partial g}{\partial t}\right) _{\varphi (p)}
\end{array}\right) 
\end{align*}
and its determinant is $\pm 4(\frac{\partial ^2g}{\partial s^2})_{\varphi (p)}(\frac{\partial g}{\partial t})_{\varphi (p)}^2$.
The critical point $p$ degenerates if and only if $(\frac{\partial ^2g}{\partial s^2})_{\varphi (p)}=0$.

On the other hand, we analyze the graphic $\varphi (S_\varphi \cap U)$ near the horizontal point $\varphi (p)$.
It is parametrized as $(s,t)=(u,0)$ and regarded as a graph of a function $g=\theta (f)$.
The derivatives of $\theta $ are
$$\frac{d\theta }{df}
=\frac{d}{df}g(u,0)
=\frac{\frac{d}{du}g(u,0)}{\frac{d}{du}f(u,0)}
=\frac{\frac{d}{du}(u)\frac{\partial g}{\partial s}(u,0)+\frac{d}{du}(0)\frac{\partial g}{\partial t}(u,0)}{\frac{d}{du}(u)\frac{\partial f}{\partial s}(u,0)+\frac{d}{du}(0)\frac{\partial f}{\partial t}(u,0)}
=\frac{\frac{\partial g}{\partial s}(u,0)}{\frac{\partial f}{\partial s}(u,0)},$$
\begin{align*}
\frac{d^2\theta }{df^2}&=\frac{d }{df}\frac{\frac{\partial g}{\partial s}(u,0)}{\frac{\partial f}{\partial s}(u,0)}\\
&=\left\{ \frac{d}{df}\left( \frac{\partial g}{\partial s}(u,0)\right) \frac{\partial f}{\partial s}(u,0)-\frac{\partial g}{\partial s}(u,0)\frac{d}{df}\left( \frac{\partial f}{\partial s}(u,0)\right) \right\} \bigg/ \left( \frac{\partial f}{\partial s}(u,0)\right) ^2\\
&=\left\{ \frac{\frac{\partial ^2g}{\partial s^2}(u,0)}{\frac{\partial f}{\partial s}(u,0)}\frac{\partial f}{\partial s}(u,0)-\frac{\partial g}{\partial s}(u,0)\frac{\frac{\partial ^2f}{\partial s^2}(u,0)}{\frac{\partial f}{\partial s}(u,0)}\right\} \bigg/ \left( \frac{\partial f}{\partial s}(u,0)\right) ^2\\
&=\left\{ \frac{\partial ^2g}{\partial s^2}(u,0)\frac{\partial f}{\partial s}(u,0)-\frac{\partial g}{\partial s}(u,0)\frac{\partial ^2f}{\partial s^2}(u,0)\right\} \bigg/ \left( \frac{\partial f}{\partial s}(u,0)\right) ^3.
\end{align*}
Noting that $(\frac{\partial g}{\partial s})_{\varphi (p)}=0$ and $(\frac{\partial f}{\partial s})_{\varphi (p)}\not=0$, the second derivative of $\theta $ at $u=0$ is $(\frac{\partial ^2g}{\partial s^2})_{\varphi (p)}/(\frac{\partial f}{\partial s})_{\varphi (p)}^2$.
The horizontal point $\varphi (p)$ is thus an inflection point if and only if $(\frac{\partial ^2g}{\partial s^2})_{\varphi (p)}=0$.
\end{component}

\begin{component}[Cusp Case]
If $p$ is a cusp point, we have the second partial derivatives:
\begin{align*}
\frac{\partial ^2G}{\partial u^2}&=\frac{\partial ^2g}{\partial s^2}+2x\frac{\partial ^2g}{\partial s\partial t}+x^2\frac{\partial ^2g}{\partial t^2},
&\frac{\partial ^2G}{\partial u\partial x}&=\frac{\partial g}{\partial t}+(u-3x^2)\frac{\partial ^2g}{\partial s\partial t}+x(u-3x^2)\frac{\partial ^2g}{\partial t^2},\\
\frac{\partial ^2G}{\partial u\partial y}&=2y\frac{\partial ^2g}{\partial s\partial t}+2xy\frac{\partial ^2g}{\partial t^2},
&\frac{\partial ^2G}{\partial x^2}&=-6x\frac{\partial g}{\partial t}+(u-3x^2)^2\frac{\partial ^2g}{\partial t^2},\\
\frac{\partial ^2G}{\partial x\partial y}&=2y(u-3x^2)\frac{\partial ^2g}{\partial t^2},
&\frac{\partial ^2G}{\partial y^2}&=2\frac{\partial g}{\partial t}+4y^2\frac{\partial ^2g}{\partial t^2}.
\end{align*}
The Hessian matrix of $G$ at $p$ is
\begin{align*}
\left( \begin{array}{ccc}
\left( \frac{\partial ^2g}{\partial s^2}\right) _{\varphi (p)}& \left( \frac{\partial g}{\partial t}\right) _{\varphi (p)}& 0\\
\left( \frac{\partial g}{\partial t}\right) _{\varphi (p)}& 0& 0\\
0& 0& 2\left( \frac{\partial g}{\partial t}\right) _{\varphi (p)}
\end{array}\right) 
\end{align*}
and its determinant is $-2(\frac{\partial g}{\partial t})_{\varphi (p)}^3\not=0$.
The critical point $p$ is thus non-degenerate.
\end{component}
\end{proof}

Apart from the question of degeneracy of the critical points, we would now like to consider the second derivatives at cusps.
We continue with the above notation in the cusp case.
We can assume that the cusp $\varphi (p)$ is not vertical, namely $(\frac{\partial f}{\partial s})_{\varphi (p)}\not=0$, after rotating if necessary.
The graphic $\varphi (S_\varphi \cap U)$ is parametrized as $(s,t)=(3x^2,2x^3)$.
The first derivative of $\varphi (S_\varphi \cap U)$ at $(3x^2,2x^3)$ for $x\not=0$ with respect to the coordinate system $(f,g)$ is
\begin{align*}
\frac{d}{df}g(3x^2,2x^3)&=\frac{\frac{d}{dx}g(3x^2,2x^3)}{\frac{d}{dx}f(3x^2,2x^3)}\\
&=\frac{\frac{d}{dx}(3x^2)\frac{\partial g}{\partial s}(3x^2,2x^3)+\frac{d}{dx}(2x^3)\frac{\partial g}{\partial t}(3x^2,2x^3)}{\frac{d}{dx}(3x^2)\frac{\partial f}{\partial s}(3x^2,2x^3)+\frac{d}{dx}(2x^3)\frac{\partial f}{\partial t}(3x^2,2x^3)}\\
&=\frac{\frac{\partial g}{\partial s}(3x^2,2x^3)+x\frac{\partial g}{\partial t}(3x^2,2x^3)}{\frac{\partial f}{\partial s}(3x^2,2x^3)+x\frac{\partial f}{\partial t}(3x^2,2x^3)}.
\end{align*}
The reader can check that the second derivative is
$$\frac{d^2}{df^2}g(3x^2,2x^3)=\frac{\frac{\partial f}{\partial s}(3x^2,2x^3)\frac{\partial g}{\partial t}(3x^2,2x^3)-\frac{\partial f}{\partial t}(3x^2,2x^3)\frac{\partial g}{\partial s}(3x^2,2x^3)+xA(x)}{6x\left\{ \frac{\partial f}{\partial s}(3x^2,2x^3)+x\frac{\partial f}{\partial t}(3x^2,2x^3)\right\} ^3},$$
where $A(x)$ is a certain smooth function.
As $x$ goes to zero, the denominator goes to zero while the numerator goes to the Jacobian of the coordinate transformation at $\varphi (p)$.
So the second derivative $\frac{d^2}{df^2}g(3x^2,2x^3)$ diverges.
Moreover, the sign of the second derivative is $\varepsilon $ when $x$ goes to zero from above, and is $-\varepsilon $ when $x$ goes to zero from below, where $\varepsilon $ is the sing of $\{ (\frac{\partial f}{\partial s})_{\varphi (p)}(\frac{\partial g}{\partial t})_{\varphi (p)}-(\frac{\partial f}{\partial t})_{\varphi (p)}(\frac{\partial g}{\partial s})_{\varphi (p)}\} /(\frac{\partial f}{\partial s})_{\varphi (p)}$.
It implies that $\varphi (p)$ is a {\it type one} cusp, that is, the tangent line at the cusp separates the two branches as on the left side of Figure \ref{fig_type}.
A {\it type two} cusp as on the right side of Figure \ref{fig_type} is also possible on a general plane curve, but on the graphic.

\begin{Lem}\label{type}
The graphic has no type two cusps.
\end{Lem}

\begin{figure}[ht]
\begin{minipage}{50pt}
\begin{center}
\includegraphics[bb=0 0 50 50]{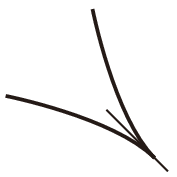}\\
type one
\end{center}
\end{minipage}
\hspace{25pt}
\begin{minipage}{50pt}
\begin{center}
\includegraphics[bb=0 0 50 50]{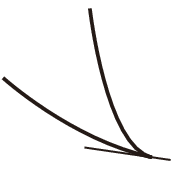}\\
type two
\end{center}
\end{minipage}
\caption{The two types of cusps on plane curves.}
\label{fig_type}
\end{figure}

We have read the existence and degeneracy of the critical points of $F,G$ from the graphic.
Lastly we would like to read the indices of non-degenerate critical points.
To do that, we mark and paint the graphic (see the pictures in Table \ref{tab_index}) in the following manner:
By the local observations in Section \ref{Stable}, each component of $S_\varphi \setminus \text{(cusp points)}$ consists of either definite fold points or indefinite fold points.
We mark each immersed arc of the graphic with the initial ``d" or ``i" according to whether it is from definite or indefinite fold points.
The image $\varphi (U)$ of a small neighborhood $U$ of a definite fold point is contained on one side of the definite arc $\varphi (S_\varphi \cap U)$.
We paint in gray the side of the collar of each definite arc in which the image is contained.

\begin{Lem}\label{index}
The index of a non-degenerate critical point of $G$ is determined by the type of the horizontal point as in Table \ref{tab_index}.
The symmetrical holds for $F$.
\end{Lem}
\begin{table}[ht]
\begin{tabular}{|c|cccc|}
\hline
index $3$&\parbox[c]{35pt}{\includegraphics[bb=0 0 35 35,origin=c,angle=180]{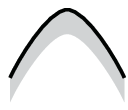}\hspace*{-6pt}\raisebox{18pt}{d}}&&&\\
\hline
index $2$&\parbox[c]{35pt}{\includegraphics[bb=0 0 35 35]{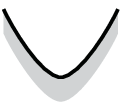}\hspace*{-12pt}\raisebox{19pt}{d}}
&\parbox[c]{35pt}{\includegraphics[bb=0 0 35 35,origin=c,angle=180]{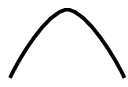}\hspace*{-5pt}\raisebox{18pt}{i}}
&\parbox[c]{35pt}{\includegraphics[bb=0 0 35 35,origin=c,angle=180]{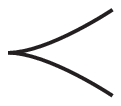}\hspace*{-26pt}\raisebox{23pt}{d}\hspace{-4pt}\raisebox{6pt}{i}}
&\parbox[c]{35pt}{\includegraphics[bb=0 0 35 35]{figures/cusp_1.eps}\hspace*{-16pt}\raisebox{23pt}{d}\hspace{-4pt}\raisebox{6pt}{i}}\\
\hline
index $1$&\parbox[c]{35pt}{\includegraphics[bb=0 0 35 35,origin=c,angle=180]{figures/concave.eps}\hspace*{-12pt}\raisebox{9pt}{d}}
&\parbox[c]{35pt}{\includegraphics[bb=0 0 35 35]{figures/horizontal.eps}\hspace*{-5pt}\raisebox{10pt}{i}}
&\parbox[c]{35pt}{\includegraphics[bb=0 0 35 35,origin=c,angle=180]{figures/cusp_1.eps}\hspace*{-26pt}\raisebox{6pt}{d}\hspace{-4pt}\raisebox{23pt}{i}}
&\parbox[c]{35pt}{\includegraphics[bb=0 0 35 35]{figures/cusp_1.eps}\hspace*{-16pt}\raisebox{6pt}{d}\hspace{-4pt}\raisebox{23pt}{i}}\\
\hline
index $0$&\parbox[c]{35pt}{\includegraphics[bb=0 0 35 35]{figures/convex.eps}\hspace*{-6pt}\raisebox{10pt}{d}}&&&\\
\hline
\end{tabular}
\vspace{8pt}
\caption{The correspondence between the index of a critical point of $G$ and the type of the horizontal point.
Here, we draw the graphic so that $f$ increases from left to right and $g$ increases from bottom to top.}
\label{tab_index}
\end{table}

\begin{proof}
We continue with the notation in the proofs of the above lemmas.
Recall that the index of a non-degenerate critical point of a function is the sum of the multiplicities of negative eigenvalues of the Hessian matrix.
See again the Hessian matrices in the fold case and the cusp case in the proof of the previous lemma.

\begin{component}[Fold Case]
The eigenvalues are $(\frac{\partial ^2g}{\partial s^2})_{\varphi (p)},2(\frac{\partial g}{\partial t})_{\varphi (p)},\pm 2(\frac{\partial g}{\partial t})_{\varphi (p)}$ as they appear in the Hessian matrix.
Recall that the double sign corresponds to the definite case and the indefinite case.
The first eigenvalue $(\frac{\partial ^2g}{\partial s^2})_{\varphi (p)}$ has the same sign as the second derivative $(\frac{\partial ^2g}{\partial s^2})_{\varphi (p)}/(\frac{\partial f}{\partial s})_{\varphi (p)}^2$ of the graphic $\varphi (S_\varphi \cap U)$ at the horizontal point $\varphi (p)$.
That is to say, it is positive if the horizontal point is downward convex, and it is negative if the horizontal point is upward convex.
In the indefinite case, it determines the index of the critical point $p$ regardless of the sign of $(\frac{\partial g}{\partial t})_{\varphi (p)}$.

Consider the sign of $(\frac{\partial g}{\partial t})_{\varphi (p)}$ in the definite case.
The form $(s,t)=(u,x^2+y^2)$ says that the $t$-axis is directed to the side of $\varphi (S_\varphi \cap U)$ in which $\varphi (U)$ is contained.
The $g$-axis is directed in the same way if $(\frac{\partial g}{\partial t})_{\varphi (p)}$ is positive, and the $g$-axis is directed against if $(\frac{\partial g}{\partial t})_{\varphi (p)}$ is negative.
It determines the index in combination with whether the horizontal point is downward or upward convex.
\end{component}

\begin{component}[Cusp Case]
The first two eigenvalues are the solutions of $\alpha $ for the equation $\alpha \left( \alpha -(\frac{\partial ^2g}{\partial s^2})_{\varphi (p)}\right) =(\frac{\partial g}{\partial t})_{\varphi (p)}^2$.
They have mutually opposite signs, and so the index of the critical point $p$ is determined by the sign of the last eigenvalue $2(\frac{\partial g}{\partial t})_{\varphi (p)}$.
Recall that one half, $\{ (3x^2,x,0)\mid x<0\} $, of $S_\varphi \cap U$ consists of definite fold points and the other half, $\{ (3x^2,x,0)\mid x>0\} $, consists of indefinite fold points.
Immersed into the plane, the indefinite arc $\{ (s,t)=(3x^2,2x^3)\mid x>0\} $ is above the definite arc $\{ (s,t)=(3x^2,2x^3)\mid x<0\} $ with respect to the $t$-axis.
With respect to the $g$-axis, the indefinite arc is above the definite arc if $(\frac{\partial g}{\partial t})_{\varphi (p)}$ is positive, and the indefinite arc is below the definite arc if $(\frac{\partial g}{\partial t})_{\varphi (p)}$ is negative.
\end{component}
\end{proof}

\section{Isotoping the graphic}\label{Isotoping}

As well as deformations of two functions induce a deformation of the product map, a deformation of a product map induces deformations of the two functions.
In this section, we study the deformations of the two functions induced by an isotopy of the graphic in the plane ${\mathbb R}^2$.

Suppose $F,G:M\rightarrow {\mathbb R}$ are Morse functions such that $\varphi :=F\times G$ is stable.
Consider a smooth ambient isotopy $\{ H_r:{\mathbb R}^2\rightarrow {\mathbb R}^2\} _{r\in [0,1]}$ of ${\mathbb R}^2$ such that $H_0=id_{{\mathbb R}^2}$.
By the definitions, $\{ H_r\circ \varphi \} _{r\in [0,1]}$ is an isotopy of $\varphi $ and consists of stable maps.
It induces homotopies $\{ \pi _f\circ H_r\circ \varphi \} _{r\in [0,1]}$ of $F$ and $\{ \pi _g\circ H_r\circ \varphi \} _{r\in [0,1]}$ of $G$.
Here $\pi _f,\pi _g:{\mathbb R}^2\rightarrow {\mathbb R}$ are the orthogonal projections.
The isotoped graphic $H_r(\varphi (S_\varphi ))$ is the graphic of $\pi _f\circ H_r\circ \varphi $ and $\pi _g\circ H_r\circ \varphi $ for each $r\in [0,1]$.

\begin{Cor}\label{isotopy}
An ambient isotopy of ${\mathbb R}^2$ induces a quasi-isotopy of $G$ if and only if it keeps the graphic without horizontal inflection points.
The same holds for $F$ by replacing ``horizontal" with ``vertical".
\end{Cor}

\begin{proof}
Lemmas \ref{critical}, \ref{degeneracy} and the definition of a quasi-isotopy lead to the corollary.
\end{proof}

We would like to develop our understanding of a horizontal inflection point.
We can think of it as the point at which either a birth or a death of a canceling pair of horizontal points occurs as illustrated in Figure \ref{fig_wave}.

\begin{figure}[ht]
\begin{minipage}{60pt}
\includegraphics[bb=0 0 60 60]{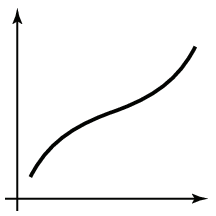}\\[-63pt]
\hspace*{-3pt}$g$\\[44pt]
\hspace*{50pt}$f$
\end{minipage}
\hspace{10pt}
{\LARGE $\leftrightarrow $}
\hspace{10pt}
\begin{minipage}{60pt}
\includegraphics[bb=0 0 60 60]{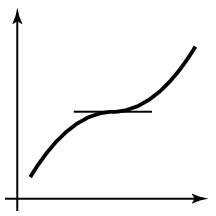}\\[-63pt]
\hspace*{-3pt}$g$\\[44pt]
\hspace*{50pt}$f$
\end{minipage}
\hspace{10pt}
{\LARGE $\leftrightarrow $}
\hspace{10pt}
\begin{minipage}{60pt}
\includegraphics[bb=0 0 60 60]{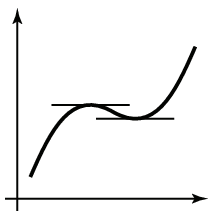}\\[-63pt]
\hspace*{-3pt}$g$\\[44pt]
\hspace*{50pt}$f$
\end{minipage}
\caption{A birth/death of a canceling pair of horizontal points.}
\label{fig_wave}
\end{figure}

\begin{Lem}\label{canceling}
A birth/death of a canceling pair of horizontal points of the graphic induces a birth/death of a canceling pair of critical points of $G$.
In particular, one on a definite arc induces one of indices $0$ and $1$ or indices $2$ and $3$, and one on an indefinite arc induces one of indices $1$ and $2$.
The same holds for $F$ by replacing ``horizontal" with ``vertical".
\end{Lem}

\begin{proof}
By Lemma \ref{critical}, a birth/death of a canceling pair of horizontal points induces a birth/death of a pair of critical points of $G$.
Moreover, by Lemma \ref{index}, the indices of the critical points are as stated.
It remains to prove that the pair of critical points is a canceling pair.

Suppose $\{ H_r\} _{r\in [0,1]}$ is an ambient isotopy of ${\mathbb R}^2$ which keeps the graphic without horizontal or vertical inflection points except for a single birth of a canceling pair of horizontal points.
The isotopy $\{ H_r\circ \varphi \} _{r\in [0,1]}$ of $\varphi $ can be regarded as a point in the mapping space $C^\infty (M\times [0,1],{\mathbb R}^2)$.
By a small deformation of $\{ H_r\} _{r\in [0,1]}$ preserving $H_0$ and $H_1$, we can assume that the third derivative of the graphic at the horizontal inflection point is not zero, and that the derivative with respect to $r$ of the slope of the graphic at the inflection point is not zero when the inflection point is horizontal.
They ensure the existence of an open neighborhood $U$ of $\{ H_r\circ \varphi \} _{r\in [0,1]}$ such that every $\{ \psi _r\} _{r\in [0,1]}$ in $U$ keeps the discriminant sets $\{ \psi _r(S_{\psi _r})\} _{r\in [0,1]}$ without horizontal or vertical inflection points except for a single birth of a canceling pair of horizontal points.

The homotopy $\{ \pi _g\circ H_r\circ \varphi \} _{r\in [0,1]}$ of $G$ can be regarded as a point in $C^\infty (M\times [0,1],{\mathbb R})$.
The map $\Phi :C^\infty (M\times [0,1],{\mathbb R}^2)\rightarrow C^\infty (M\times [0,1],{\mathbb R})\times C^\infty (M\times [0,1],{\mathbb R})$, $\{ \psi _r\} _{r\in [0,1]}\mapsto (\{ \pi _f\circ \psi _r\} _{r\in [0,1]},\{ \pi _g\circ \psi _r\} _{r\in [0,1]})$ is homeomorphic by \cite[Chapter I\hspace{-.1em}I, Proposition 3.6]{golubitsky-guillemin}.
The subset $(\Pi _g\circ \Phi )(U)\subset C^\infty (M\times [0,1],{\mathbb R})$ is therefore an open neighborhood of $\{ \pi _g\circ H_r\circ \varphi \} _{r\in [0,1]}$, where $\Pi _g$ is the projection to the second factor.

It is well known that any homotopy of a smooth function can be approximated by a homotopy $\{ \xi _r\} _{r\in [0,1]}$ such that $\xi _r$ is Morse for all but finitely many $r\in [0,1]$ and either a birth or a death of a canceling pair of critical points occurs at each of the finitely many $r$.
So there exists such a homotopy $\{ \xi _r\} _{r\in [0,1]}$ in $(\Pi _g\circ \Phi )(U)$.
By the definition of $U$ and again Lemma \ref{critical}, $\{ \xi _r\} _{r\in [0,1]}$ is a quasi-isotopy except for a single birth of a canceling pair of critical points.
It implies that the pair of critical points is originally a canceling pair.
\end{proof}

We conclude this section with some remarks about how cusps of the graphic can be isotoped.
Lemma \ref{type} implies a bit anti-intuitive fact that no smooth ambient isotopy of the plane can curl a cusp of the graphic to be type two.
In other words, no smooth ambient isotopy can take an inflection point to a cusp because second derivatives of the graphic diverge at cusps.
The number of horizontal points is preserved near a horizontal cusp as illustrated in Figure \ref{fig_swing}, and the critical point of $G$ remains non-degenerate by Lemma \ref{degeneracy}.
In fact, the cusp case in Lemma \ref{index} can also be understood by this move.

\begin{figure}[ht]
\begin{minipage}{60pt}
\includegraphics[bb=0 0 60 60]{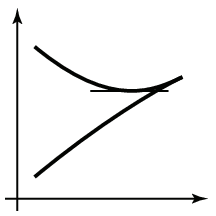}\\[-63pt]
\hspace*{-3pt}$g$\\[44pt]
\hspace*{50pt}$f$
\end{minipage}
\hspace{10pt}
{\LARGE $\leftrightarrow $}
\hspace{10pt}
\begin{minipage}{60pt}
\includegraphics[bb=0 0 60 60]{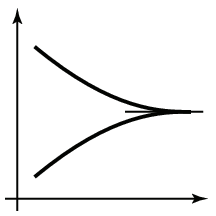}\\[-63pt]
\hspace*{-3pt}$g$\\[44pt]
\hspace*{50pt}$f$
\end{minipage}
\hspace{10pt}
{\LARGE $\leftrightarrow $}
\hspace{10pt}
\begin{minipage}{60pt}
\includegraphics[bb=0 0 60 60]{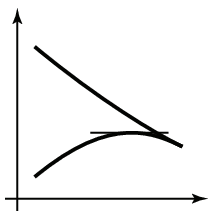}\\[-63pt]
\hspace*{-3pt}$g$\\[44pt]
\hspace*{50pt}$f$
\end{minipage}
\caption{A typical move involving a horizontal cusp.}
\label{fig_swing}
\end{figure}

\section{Proof of the theorems}\label{Proof}

By Johnson's theorem and Lemma \ref{type}, the Reidemeister--Singer distance between two Heegaard splittings is at most the number of indefinite negative slope inflection points of the graphic.
We would like to deform the graphic to reduce this number with control on the deformations of the Morse functions.
However, it seems a complicated problem what kind of deformation effectively reduces the number of inflection points themselves.
We therefore consider a deformation of the graphic which makes the slopes of inflection points positive.

Suppose $(\Sigma ,V^-,V^+),(T,W^-,W^+)$ are Heegaard splittings for a closed orientable connected smooth $3$-manifold $M$.
Let $F,G:M\rightarrow {\mathbb R}$ be Morse functions representing $(\Sigma ,V^-,V^+),(T,W^-,W^+)$, respectively.
We can choose each of them so that it has unique critical points of indices zero and three, and we can assume that $F\times G$ is stable by Lemma \ref{stable}.

We deform the graphic of $F,G$ by an ambient isotopy $\{ H_r\} _{r\in [g_-,g_+]}$ of ${\mathbb R}^2$ defined as follows:
Let $g_-$ be a value below the minimum value of $G$, and let $g_+$ be a value above the maximum value of $G$.
Choose $\delta $ to be a sufficiently small positive constant and $\Delta $ to be a sufficiently large constant.
Let $\{ h_r:{\mathbb R}\rightarrow {\mathbb R}\} _{r\in [g_-,g_+]}$ be a smooth family of monotonously increasing smooth functions such that $h_r(g)=\Delta (g-r)$ if $g\leq r-\delta $ and $h_r(g)=0$ if $g\geq r+\delta $.
We define the isotopy $\{ H_r\} _{r\in [g_-,g_+]}$ by $H_r(f,g):=(f+h_r(g)-h_{g_-}(g),g)$.
Let $\Gamma $ denote the graphic of $F,G$ and $\Gamma _r=H_r(\Gamma )$.

\begin{figure}[ht]
$\cdots $ {\LARGE $\rightarrow $}
\hspace{10pt}
\begin{minipage}{100pt}
\includegraphics[bb=0 0 100 120]{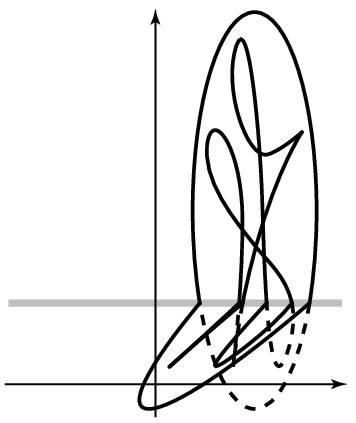}\\[-120pt]
\hspace*{37pt}$g$\\[59pt]
\hspace*{38pt}$r$\\[23pt]
\hspace*{88pt}$f$
\end{minipage}
\hspace{10pt}
{\LARGE $\rightarrow $}
\hspace{10pt}
\begin{minipage}{100pt}
\includegraphics[bb=0 0 100 120]{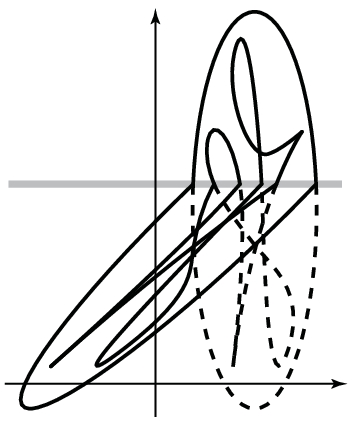}\\[-120pt]
\hspace*{37pt}$g$\\[25pt]
\hspace*{38pt}$r$\\[57pt]
\hspace*{88pt}$f$
\end{minipage}
\hspace{10pt}
{\LARGE $\rightarrow $} $\cdots $
\caption{The outline of the deformation of the graphic $\Gamma _r$ as $r$.
The original graphic $\Gamma =\Gamma _{g_-}$ is shown with broken lines.}
\label{fig_shear}
\end{figure}

The isotopy $\{ H_r\} _{r\in [g_-,g_+]}$ shears the graphic by the process as in Figure \ref{fig_shear}.
Note that the graphic $\Gamma $ of $F,G$ is contained in the image $(F\times G)(M)$ and so in the region $\{ (f,g)\in {\mathbb R}^2\mid g_-<g<g_+\} $.
The thin band $b_r:=\{ r-\delta \leq g\leq r+\delta \} $ looks like a scanning line which runs from below to above.
In the front region $\{ g>r+\delta \} $, the graphic $\Gamma _r$ remains unchanged from $\Gamma $.
In the back region $\{ g<r-\delta \} $, the graphic $\Gamma _r$ is the result of shearing $\Gamma $ and is translated leftward.
Since the shearing slope $\Delta $ is sufficiently large, $\Gamma _r\cap \{ g<r-\delta \} $ has positive slope outside of small neighborhoods of horizontal points.
In particular, every inflection point of it has positive slope.

This isotopy induces a homotopy $\{ F_r:=\pi _f\circ H_r\circ (F\times G)\} _{r\in [g_-,g_+]}$ of $F$.
When $F_r$ is Morse, it represents a Heegaard splitting $(\Sigma _r,V_r^-,V_r^+)$.
On the other hand, $\pi _g\circ H_r\circ (F\times G)$ is constantly $G$ because $H_r$ preserves the second coordinate $g$.
In particular, the represented Heegaard splitting $(T,W^-,W^+)$ is preserved.
Since every inflection point of the result graphic $\Gamma _{g_+}$ has positive slope, the function $F_{g_+}$ is Morse and the Heegaard splitting $(\Sigma _{g_+},V_{g_+}^-,V_{g_+}^+)$ is isotopic to $(T,W^-,W^+)$ by Theorem \ref{johnson} and Lemma \ref{type}.
To bound the Reidemeister--Singer distance between $(\Sigma ,V^-,V^+)=(\Sigma _{g_-},V_{g_-}^-,V_{g_-}^+)$ and $(T,W^-,W^+)$, we analyze how $(\Sigma _r,V_r^-,V_r^+)$ changes as $r$.

We observe the deformation of the graphic $\Gamma _r$, paying attention to births and deaths of vertical points.
No births or deaths happen in the front and back region, but in the scanning line $b_r$.
Note that we can make the deformation in $b_r$ arbitrarily sharp by choosing $\delta $ small and $\Delta $ large.
Note also that $\Gamma $ has only finitely many horizontal points, vertical points and cusps.
We can therefore assume that the figures below do not lose generality.

\begin{figure}[ht]
\parbox[c]{80pt}{\includegraphics[bb=0 0 80 80]{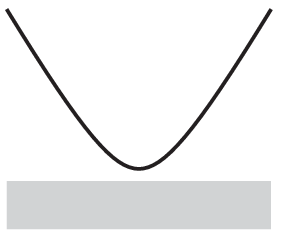}}
\hspace{10pt}
{\LARGE $\rightarrow $}
\hspace{10pt}
\parbox[c]{80pt}{\includegraphics[bb=0 0 80 80]{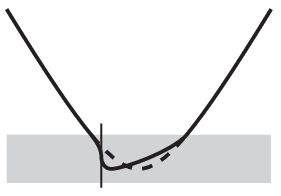}}
\hspace{10pt}
{\LARGE $\rightarrow $}
\hspace{10pt}
\parbox[c]{80pt}{\includegraphics[bb=0 0 80 80]{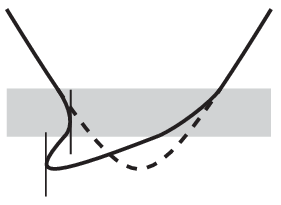}}
\caption{The deformation of $\Gamma _r$ when $b_r$ passes a downward convex horizontal point of $\Gamma $.
The band $b_r$ is shown in gray.
It looks pretty thick because the picture has been greatly enlarged.}
\label{fig_birth_h}
\end{figure}

When $b_r$ passes a downward convex horizontal point of $\Gamma $, a birth of a canceling pair of vertical points occurs as in Figure \ref{fig_birth_h}.
By Lemmas \ref{stabilization} and \ref{canceling}, it preserves $(\Sigma _r,V_r^-,V_r^+)$ if the horizontal point is definite, and it causes a stabilization for $(\Sigma _r,V_r^-,V_r^+)$ if the horizontal point is indefinite.

\begin{figure}[ht]
\parbox[c]{80pt}{\includegraphics[bb=0 0 80 80]{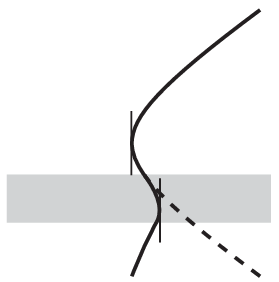}}
\hspace{10pt}
{\LARGE $\rightarrow $}
\hspace{10pt}
\parbox[c]{80pt}{\includegraphics[bb=0 0 80 80]{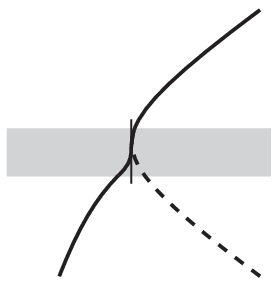}}
\hspace{10pt}
{\LARGE $\rightarrow $}
\hspace{10pt}
\parbox[c]{80pt}{\includegraphics[bb=0 0 80 80]{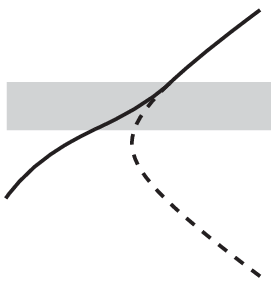}}
\caption{At a leftward convex vertical point of $\Gamma $.}
\label{fig_death_v}
\end{figure}

When $b_r$ passes a leftward convex vertical point of the original graphic $\Gamma $, a death of a canceling pair of vertical points occurs as in Figure \ref{fig_death_v}.
It preserves $(\Sigma _r,V_r^-,V_r^+)$ if the vertical point is definite, and it causes a destabilization for $(\Sigma _r,V_r^-,V_r^+)$ if indefinite.

\begin{figure}[ht]
\parbox[c]{80pt}{\includegraphics[bb=0 0 80 80]{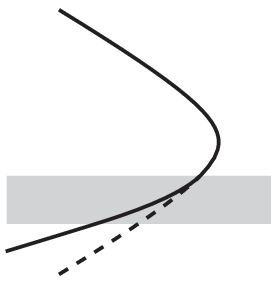}}
\hspace{10pt}
{\LARGE $\rightarrow $}
\hspace{10pt}
\parbox[c]{80pt}{\includegraphics[bb=0 0 80 80]{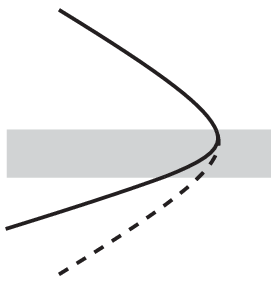}}
\hspace{10pt}
{\LARGE $\rightarrow $}
\hspace{10pt}
\parbox[c]{80pt}{\includegraphics[bb=0 0 80 80]{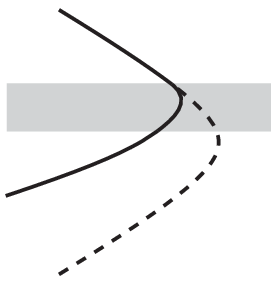}}
\caption{At a rightward convex vertical point of $\Gamma $.}
\label{fig_through_v}
\end{figure}
\begin{figure}[ht]
\parbox[c]{80pt}{\includegraphics[bb=0 0 80 80]{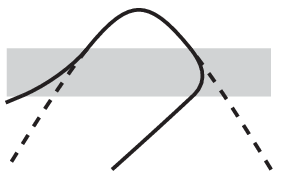}}
\hspace{10pt}
{\LARGE $\rightarrow $}
\hspace{10pt}
\parbox[c]{80pt}{\includegraphics[bb=0 0 80 80]{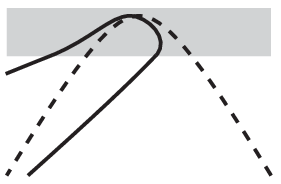}}
\hspace{10pt}
{\LARGE $\rightarrow $}
\hspace{10pt}
\parbox[c]{80pt}{\includegraphics[bb=0 0 80 80]{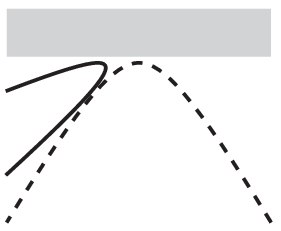}}
\caption{At an upward convex horizontal point of $\Gamma $.}
\label{fig_through_h}
\end{figure}

As in Figures \ref{fig_through_v} and \ref{fig_through_h}, no vertical inflection points appear at a rightward convex vertical point and an upward convex horizontal point of $\Gamma $, and so $(\Sigma _r,V_r^-,V_r^+)$ is preserved.

\begin{figure}[ht]
\parbox[c]{80pt}{\includegraphics[bb=0 0 80 80]{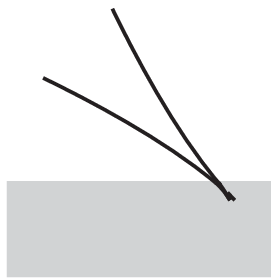}}
\hspace{10pt}
{\LARGE $\rightarrow $}
\hspace{10pt}
\parbox[c]{80pt}{\includegraphics[bb=0 0 80 80]{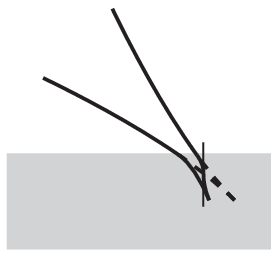}}\\[8pt]
\hspace{10pt}
{\LARGE $\rightarrow $}
\hspace{10pt}
\parbox[c]{80pt}{\includegraphics[bb=0 0 80 80]{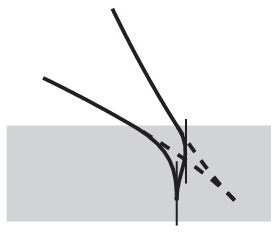}}
\hspace{10pt}
{\LARGE $\rightarrow $}
\hspace{10pt}
\parbox[c]{80pt}{\includegraphics[bb=0 0 80 80]{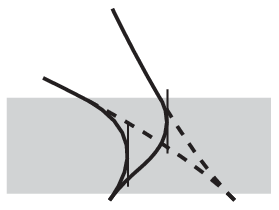}}
\caption{At a downer right pointing cusp of $\Gamma $.}
\label{fig_birth_c}
\end{figure}

When $b_r$ passes a downer right pointing cusp of $\Gamma $, a birth of a canceling pair of vertical points occurs as in Figure \ref{fig_birth_c}.
Note that the cusp remains to be type one as remarked in Section \ref{Isotoping}, and a vertical inflection point appears on the right arc.
It preserves $(\Sigma _r,V_r^-,V_r^+)$ if the right arc is definite, and it causes a stabilization for $(\Sigma _r,V_r^-,V_r^+)$ if the right arc is indefinite.

\begin{figure}[ht]
\parbox[c]{80pt}{\includegraphics[bb=0 0 80 80]{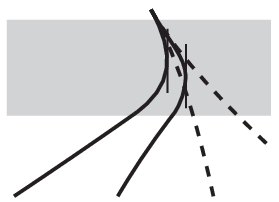}}
\hspace{10pt}
{\LARGE $\rightarrow $}
\hspace{10pt}
\parbox[c]{80pt}{\includegraphics[bb=0 0 80 80]{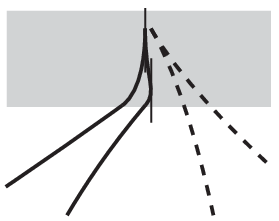}}\\[8pt]
\hspace{10pt}
{\LARGE $\rightarrow $}
\hspace{10pt}
\parbox[c]{80pt}{\includegraphics[bb=0 0 80 80]{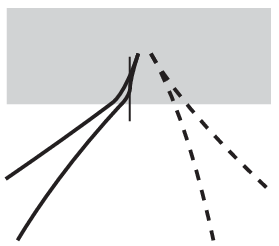}}
\hspace{10pt}
{\LARGE $\rightarrow $}
\hspace{10pt}
\parbox[c]{80pt}{\includegraphics[bb=0 0 80 80]{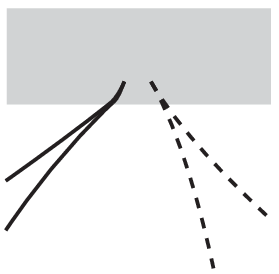}}
\caption{At an upper left pointing cusp of $\Gamma $.}
\label{fig_death_c}
\end{figure}

When $b_r$ passes an upper left pointing cusp of $\Gamma $, a canceling pair of vertical points deaths as in Figure \ref{fig_death_c}.
It preserves $(\Sigma _r,V_r^-,V_r^+)$ if the right arc is definite, and it causes a destabilization for $(\Sigma _r,V_r^-,V_r^+)$ if the right arc is indefinite.

Consider what happens at a downer left pointing cusp and an upper right pointing cusp of $\Gamma $.
The reader can check that they keep the graphic without vertical inflection points, and so preserve $(\Sigma _r,V_r^-,V_r^+)$.

Note that the graphic $\Gamma $ possibly has vertical or horizontal cusps.
The reader can check that $(\Sigma _r,V_r^-,V_r^+)$ has a stabilization at a right pointing horizontal cusp where the upper arc is indefinite, a destabilization at an upper pointing vertical cusp where the right arc is indefinite, and nothing at the other types of vertical or horizontal cusps.
In fact, those can also be understood by the move in Figure \ref{fig_swing}.

Note that in the remaining parts of $\Gamma $ including crossing points and inflection points, no vertical inflection points appear, and so $(\Sigma _r,V_r^-,V_r^+)$ is preserved.
Thus, $(\Sigma _{g_+},V_{g_+}^-,V_{g_+}^+)$ is obtained from $(\Sigma ,V^-,V^+)$ by
\begin{align*}
&\sharp \left\{ 
\parbox[c]{30pt}{\includegraphics[bb=0 0 35 35,width=30pt]{figures/horizontal.eps}\hspace{-5pt}\raisebox{9pt}{i}},
\parbox[c]{30pt}{\includegraphics[bb=0 0 35 35,width=30pt]{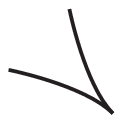}\hspace{-20pt}\raisebox{1pt}{d}\hspace{7pt}\raisebox{15pt}{i}},
\parbox[c]{30pt}{\includegraphics[bb=0 0 35 35,width=30pt]{figures/cusp_1.eps}\hspace{-15pt}\raisebox{20pt}{i}\hspace{-4pt}\raisebox{4pt}{d}}
\text{ on }\Gamma \right\} \text{ times stabilizations and}\\
&\sharp \left\{ 
\parbox[c]{30pt}{\includegraphics[bb=0 0 35 35,height=30pt,origin=c,angle=-90]{figures/horizontal.eps}\hspace{-17pt}\raisebox{9pt}{i}},
\parbox[c]{30pt}{\includegraphics[bb=0 0 35 35,width=30pt,origin=c,angle=180]{figures/cusp_2.eps}\hspace{-27pt}\raisebox{9pt}{d}\hspace{7pt}\raisebox{23pt}{i}},
\parbox[c]{30pt}{\includegraphics[bb=0 0 35 35,height=30pt,origin=c,angle=90]{figures/cusp_1.eps}\hspace{-25pt}\raisebox{16pt}{d}\hspace{10pt}\raisebox{16pt}{i}}
\text{ on }\Gamma \right\} \text{ times destabilizations.}
\end{align*}

By Lemma \ref{index} and the assumption that each of $F,G$ has unique critical points of indices zero and three,
\begin{align*}
&\sharp \left\{ 
\parbox[c]{30pt}{\includegraphics[bb=0 0 35 35,height=30pt,origin=c,angle=-90]{figures/horizontal.eps}\hspace{-17pt}\raisebox{9pt}{i}},
\parbox[c]{30pt}{\includegraphics[bb=0 0 35 35,height=30pt,origin=c,angle=90]{figures/cusp_1.eps}\hspace{-25pt}\raisebox{16pt}{d}\hspace{10pt}\raisebox{16pt}{i}}
\text{ on }\Gamma \right\} \leq \sharp \{ \text{index }1\text{ critical point of }F\} ={\rm g}(\Sigma ),\\
&\sharp \left\{ 
\parbox[c]{30pt}{\includegraphics[bb=0 0 35 35,width=30pt]{figures/horizontal.eps}\hspace{-5pt}\raisebox{9pt}{i}},
\parbox[c]{30pt}{\includegraphics[bb=0 0 35 35,width=30pt]{figures/cusp_1.eps}\hspace{-15pt}\raisebox{20pt}{i}\hspace{-4pt}\raisebox{4pt}{d}}
\text{ on }\Gamma \right\}\leq \sharp \{ \text{index }1\text{ critical point of }G\} ={\rm g}(T).
\end{align*}
The Reidemeister--Singer distance $d_+$ between $(\Sigma ,V^-,V^+),(T,W^-,W^+)$ therefore satisfies
\begin{align}
d_+&\leq \sharp \left\{ 
\parbox[c]{30pt}{\includegraphics[bb=0 0 35 35,width=30pt]{figures/horizontal.eps}\hspace{-5pt}\raisebox{9pt}{i}},
\parbox[c]{30pt}{\includegraphics[bb=0 0 35 35,width=30pt]{figures/cusp_2.eps}\hspace{-20pt}\raisebox{1pt}{d}\hspace{7pt}\raisebox{15pt}{i}},
\parbox[c]{30pt}{\includegraphics[bb=0 0 35 35,width=30pt]{figures/cusp_1.eps}\hspace{-15pt}\raisebox{20pt}{i}\hspace{-4pt}\raisebox{4pt}{d}},
\parbox[c]{30pt}{\includegraphics[bb=0 0 35 35,height=30pt,origin=c,angle=-90]{figures/horizontal.eps}\hspace{-17pt}\raisebox{9pt}{i}},
\parbox[c]{30pt}{\includegraphics[bb=0 0 35 35,width=30pt,origin=c,angle=180]{figures/cusp_2.eps}\hspace{-27pt}\raisebox{9pt}{d}\hspace{7pt}\raisebox{23pt}{i}},
\parbox[c]{30pt}{\includegraphics[bb=0 0 35 35,height=30pt,origin=c,angle=90]{figures/cusp_1.eps}\hspace{-25pt}\raisebox{16pt}{d}\hspace{10pt}\raisebox{16pt}{i}}
\text{ on }\Gamma \right\} \notag \\
&\leq {\rm g}(\Sigma )+{\rm g}(T)+\sharp \left\{ 
\parbox[c]{30pt}{\includegraphics[bb=0 0 35 35,width=30pt]{figures/cusp_2.eps}\hspace{-20pt}\raisebox{1pt}{d}\hspace{7pt}\raisebox{15pt}{i}},
\parbox[c]{30pt}{\includegraphics[bb=0 0 35 35,width=30pt,origin=c,angle=180]{figures/cusp_2.eps}\hspace{-27pt}\raisebox{9pt}{d}\hspace{7pt}\raisebox{23pt}{i}}
\text{ on }\Gamma \right\} .\label{eq1}
\end{align}

Consider another ambient isotopy $\{ H'_r\} _{r\in [-g_+,-g_-]}$ defined by $H'_r(f,g):=(f-h_r(-g)+h_{-g_+}(-g),g)$.
It shears the graphic positively as well as $\{ H_r\} _{r\in [g_-,g_+]}$, but the scanning line runs from above to below.
By similar observations,
\begin{equation}
d_+\leq {\rm g}(\Sigma )+{\rm g}(T)+\sharp \left\{ 
\parbox[c]{30pt}{\includegraphics[bb=0 0 35 35,width=30pt]{figures/cusp_2.eps}\hspace{-18pt}\raisebox{1pt}{i}\hspace{7pt}\raisebox{15pt}{d}},
\parbox[c]{30pt}{\includegraphics[bb=0 0 35 35,width=30pt,origin=c,angle=180]{figures/cusp_2.eps}\hspace{-25pt}\raisebox{9pt}{i}\hspace{7pt}\raisebox{23pt}{d}}
\text{ on }\Gamma \right\} .\label{eq2}
\end{equation}
By $((\ref{eq1})+(\ref{eq2}))/2$,
\begin{align}
d_+&\leq {\rm g}(\Sigma )+{\rm g}(T)+\sharp \{ \text{negative slope cusp of }\Gamma \} /2\label{eq3}\\
&\leq {\rm g}(\Sigma )+{\rm g}(T)+c(F\times G)/2\notag 
\end{align}
to conclude the proof of Theorem \ref{main-1}.

Consider similar ambient isotopies of ${\mathbb R}^2$ shearing the graphic negatively.
The sheared graphic has no positive slope inflection points.
Then consider the reflection of ${\mathbb R}^2$ in the $f$-axis, which makes the graphic without negative slope inflection points.
It corresponds to replacing the Morse function $G$ to $-G$ and flipping the Heegaard splitting $(T,W^-,W^+)$ to be $(T,W^+,W^-)$.
By similar arguments, the Reidemeister--Singer distance $d_-$ between $(\Sigma ,V^-,V^+),(T,W^+,W^-)$ satisfies
\begin{equation}
d_-\leq {\rm g}(\Sigma )+{\rm g}(T)+\sharp \{ \text{positive slope cusp of }\Gamma \} /2.\label{eq4}
\end{equation}
By $((\ref{eq3})+(\ref{eq4}))/2$, the Reidemeister--Singer distance between $\Sigma ,T$ is
\begin{equation*}
\min \{ d_+,d_-\} \leq (d_++d_-)/2\leq {\rm g}(\Sigma )+{\rm g}(T)+c(F\times G)/4
\end{equation*}
to conclude the proof of Theorem \ref{main-2}.

\bibliographystyle{amsplain}

\end{document}